\documentclass[letterpaper,11pt]{amsart}

\usepackage[margin=1.19in]{geometry}
\usepackage{amsmath,amsthm,amssymb,mathtools}
\usepackage{xspace,xcolor}
\usepackage[breaklinks,colorlinks,citecolor=teal,linkcolor=teal,urlcolor=teal,pagebackref,hyperindex]{hyperref}
\usepackage[alphabetic]{amsrefs}
\usepackage[all]{xy}

\theoremstyle{plain}
\newtheorem{thm}{Theorem}[section]
\newtheorem{lem}[thm]{Lemma}
\newtheorem{prop}[thm]{Proposition}
\newtheorem{cor}[thm]{Corollary}

\theoremstyle{definition}

\newtheorem{eg}[thm]{Example}

\theoremstyle{remark}
\newtheorem{rmk}[thm]{Remark}

\numberwithin{equation}{section}

\def\N{{\mathbb N}}
\def\Z{{\mathbb Z}}
\def\Q{{\mathbb Q}}
\def\R{{\mathbb R}}
\def\C{{\mathbb C}}

\def\A{{\mathbb A}}

\def\cH{\mathcal{H}}

\def\cL{\mathcal{L}}

\def\I{\mathcal{I}}
\def\J{\mathcal{J}}
\def\O{\mathcal{O}}

\def\fa{\mathfrak{a}}

\def\a{\alpha}
\def\b{\beta}
\def\g{\gamma}

\def\f{\phi}
\def\ff{\psi}
\def\e{\eta}

\def\k{\kappa}
\def\l{\lambda}
\def\n{\nu}

\def\p{\pi}
\def\r{\rho}
\def\s{\sigma}
\def\t{\tau}
\def\x{\xi}

\def\D{\Delta}
\def\G{\Gamma}

\def\Om{\Omega}

\let\oldhat\^
\let\oldtilde\~

\def\.{\cdot}
\def\^{\widehat}
\def\~{\widetilde}
\def\o{\circ}
\def\ov{\overline}

\def\rat{\dashrightarrow}

\def\inj{\hookrightarrow}

\def\de{\partial}

\def\liminv{\underleftarrow{\lim}}

\renewcommand{\and}{ \ \ \text{ and } \ \ }

\def\sm{\mathrm{sm}}
\def\sing{\mathrm{sing}}

\def\ver{\mathrm{exc}}
\def\hor{\mathrm{hor}}
\def\prim{\mathrm{prim}}

\DeclareMathOperator{\Spec} {Spec}

\DeclareMathOperator{\Ex} {Ex}

\DeclareMathOperator{\ord} {ord}

\DeclareMathOperator{\Hom} {Hom}

\DeclareMathOperator{\Min} {Min}
\DeclareMathOperator{\Ter} {Ter}

\begin{document}

\title{Terminal valuations and the Nash problem}

\author{Tommaso de Fernex}
\address[T.\ de Fernex]{Department of Mathematics, University of Utah,
155 South 1400 East, Salt Lake City, UT 48112-0090, 
USA}
\email{defernex@math.utah.edu}

\author{Roi Docampo}
\address[R.\ Docampo]{Instituto de Matem\'{a}tica e Estat\'{\i}stica, 
Universidade Federal Fluminense,
Rua M\'{a}rio Santos Braga, s/n, 24020-140 Niter\'{o}i, RJ, 
Brasil}
\email{roi@mat.uff.br}

\thanks{The research of the first author was partially supported by NSF
CAREER grant DMS-0847059 and NSF FRG Grant DMS-1265285.
The research of the second author was partially supported by
CNPq/Ci\oldhat{e}ncia Sem Fronteiras, under the program Atra\c{c}\oldtilde{a}o de
Jovens Talentos, Processo 370329/2013--9.
}



\begin{abstract}
Let $X$ be an algebraic variety of characteristic zero. Terminal valuations
are defined in the sense of the minimal model program, 
as those valuations given by the exceptional divisors on a minimal
model over $X$. We prove that every terminal valuation over $X$
is in the image of the Nash map, and thus it corresponds to a
maximal family of arcs through the singular locus of $X$.
In dimension two, this result gives a new proof of the theorem of
Fern\'andez de Bobadilla and Pe Pereira stating that, 
for surfaces, the Nash map is a bijection. 
\end{abstract}

\maketitle

\thispagestyle{empty}


\section{Introduction}

Working in characteristic zero, the space of formal arcs 
passing through the singular points of an algebraic variety $X$
decomposes into finitely many irreducible families, 
and carries some of the essential information encoded in a resolution of singularities.
The \emph{Nash map} associates a divisorial valuation to 
every maximal irreducible family of arcs through the singular locus
of $X$ \cite{Nas95}. 
In this paper, we will refer to these valuations as the \emph{Nash valuations} over $X$.

The Nash problem asks for a geometric characterization of Nash valuations
in terms of resolutions of $X$. To this end, Nash
introduced the notion of \emph{essential valuations}
as those divisorial valuations whose center on every resolution
is an irreducible component of the inverse image of the singular locus of $X$. 
It is easy to see that every Nash valuation is essential, and Nash asked whether
the converse is also true. Regarded as a function 
to the set of essential valuations, the question is whether 
the Nash map is surjective. 

The Nash problem was successfully settled in dimension two in \cite{FdBPP12}
using topological arguments
(we refer to their paper for a comprehensive list
of references on previous results).
In higher dimensions, the characterization of Nash valuations as essential valuations 
is known to hold for toric singularities and in some other special 
cases~\cite{IK03,Ish05,Ish06,GP07,PPP08,LJR12,LA}.
However, examples showing that the Nash map is not always surjective
were found in all dimensions $\ge 3$~\cite{IK03,dF13,JK}.
In view of these examples, one should rephrase the problem 
by asking whether there is some other way to characterize the image of the Nash map. 

In this paper, we approach this problem from the point of view of
the minimal model program.
Recall that a minimal model over $X$ is a projective birational
morphism $f \colon Y \to X$ from a normal variety $Y$ with terminal singularities
such that the canonical class $K_Y$ is relatively nef over $X$.
We say that a divisorial valuation $\n$ on $X$ is 
\emph{terminal with respect to the minimal model program over $X$}, 
or simply that it is a \emph{terminal valuation} over $X$, 
if there exists a prime exceptional divisor $E$ on a minimal model $f \colon Y \to X$ 
such that $\n = \ord_E$.

Our main result provides a characterization of a subset of the image of the Nash map
as the one consisting of terminal valuations.

\begin{thm}
\label{t:main}
Every terminal valuation over $X$ is a Nash valuation. 
\end{thm}

Since, in dimension two, terminal valuations and essential valuations are clearly the same
(both are the valuations defined by the exceptional divisors on the minimal
resolution), we obtain a new, purely algebro-geometric proof of the main theorem of \cite{FdBPP12}.

\begin{cor}[\cite{FdBPP12}]
\label{c:main-dim2}
The Nash map is a bijection in dimension two.
\end{cor}

Theorem~\ref{t:main}
is the natural generalization of this result to higher dimensions. 
It implies, for instance, that the Nash map is surjective whenever 
there exists a nonsingular minimal model over $X$ with exceptional locus
of pure codimension one.
Since every exceptional divisor over a variety with terminal singularities
is uniruled \cite{HM07}, Theorem~\ref{t:main} also implies, as a by-product, that
every divisorial valuation defined by a divisor that is not uniruled
is necessarily a Nash valuation. This recovers, in particular, 
the main result of \cite{LJR12}, whose proof is however simpler and more direct. 

At a first sight, one may wonder whether all Nash valuations are terminal valuations. 
While this is the case in dimension two, it fails for simple reasons in higher dimensions.
For instance, it is clear that there are no terminal valuations over
a variety with terminal singularities, or over a variety which admits a small resolution. 
On the other hand, there are always Nash valuations over any singular variety. 

More examples showing that not all Nash valuations are terminal valuations can be constructed 
using toric geometry. In \S\ref{s:toric}, we give a toric description
of terminal valuations over a toric variety and compare it with
the description of Nash valuations given in \cite{IK03}.

In conclusion, Theorem~\ref{t:main} should be viewed as complementing
the fact that Nash valuations are essential. We have inclusions
\[
\{\text{terminal valuations}\} \subset 
\{\text{Nash valuations}\} \subset 
\{\text{essential valuations}\},
\]
but either inclusion can be strict.

\subsection{Outline of the proof}
\label{s:outline}

We were led to consider minimal models (and define, accordingly, terminal valuations)
as a result of our attempt to understand, from
an algebro-geometric standpoint, some of the topological computations carried
out in \cite{FdBPP12}.
The idea of looking at divisors on minimal models in connection to the Nash problem
was also suggested by Fern\'andez de Bobadilla.

For simplicity, let us focus on the two-dimensional case, as 
the main ideas of the proof are already there.
So, let $X$ be a surface. Let $f \colon Y \to X$ be the minimal resolution, 
and let $f_\infty \colon Y_\infty \to X_\infty$ be the map induced 
on arc spaces.
For every irreducible component $E$ of the exceptional locus $\Ex(f)$, 
let $E^\o \subset E$ denote 
the largest open set disjoint from the other irreducible components of $\Ex(f)$.
Then let $N_E^\o := f_\infty(\p_Y^{-1}(E^\o))$
where $\p_Y \colon Y_\infty \to Y$ is the natural projection, and let
$N_E \subset X_\infty$ be the closure of $N_E^\o$. 

Assuming that the Nash problem fails for $X$, one can find two irreducible components
$E$ and $F$ of $\Ex(f)$ such that $N_E \subsetneq N_F$.
The idea, due to Lejeune-Jalabert~\cite{LJ80}, 
is to detect this by producing a morphism $\Phi \colon \Spec K[[s]] \to X_\infty$, 
for some field extension $K/k$, such that 
\[
\Phi(0) \in N_E^\o
\and
\Phi(\e) \in N_F \smallsetminus N_E.
\]
Here $0$ is the closed point of $\Spec K[[s]]$, and $\e$ is its generic point.
Such a map is called a \emph{wedge}.
The existence of the wedge $\Phi$ is a delicate issue, and is established in \cite{Reg06}.
Note that $\Phi$ does not factor through $Y_\infty$.

Let us assume that $K = k$. 
The wedge $\Phi$ can be regarded as a morphism 
\[
\Phi \colon S = \Spec k[[s,t]] \to X
\]
that does not lift to $Y$.

At this point, the approach in \cite{FdBPP12} is roughly the following. 
Reducing to work over $\C$, 
and using a suitable approximation theorem, one can assume that $\Phi$ is locally given 
by power series with positive radius of convergence, and thus assume without loss of generality that
$S \subset \C^2$ is the product of two open disks of radius 1
\cite{FdB12}. The curves $(s=\l)\subset S$, 
for $0 < |\l| < 1$, lift to $Y$ and degenerate, as $\l \to 0$, to a cycle 
$\sum a_i E_i + T'$ that
is supported on the union of the exceptional locus $\Ex(f)$ of $f$ and the image $T'$
of the $t$-axis $(s=0) \subset S$. 
The contradiction is then reached by computing the Euler
characteristic of these curves in two ways, as images of small disks, 
and as they degenerate inside a small tubular neighborhood of $\Ex(f) \cup T'$.
The contradiction, resulting from the computation, stems from 
the fact $\Ex(f)$ does not contain any rational curve with self-intersection $-1$.

In order to translate this into algebro-geometric language, we 
take a resolution of indeterminacy of the rational map $f^{-1} \o \Phi \colon S \rat Y$. 
We construct the resolution by taking a minimal sequence of blow-ups of maximal ideals. 
This gives us a diagram
\[
\xymatrix{
Z \ar[r]^\f \ar[d]_g & Y \ar[d]^f \\
S \ar[r]^\Phi & X
}
\]
where $Z$ is a regular two-dimensional scheme. 
Then we shift the computation from $Y$ to $Z$. 

This reduction has several advantages.
First, it allows us to bypass the use of approximation theorems and to work 
directly in the formal setting.
Furthermore, we avoid having to deal with the singularities of $\Ex(f)$ and can
work in fact with partial resolutions $Y \to X$. 
Finally, working on $Z$ allows us to extend the computation to all higher dimensions, by taking 
wedges defined over suitable field extensions $K/k$.

The proof relies on the analysis of the ramification of the map $\f$
at the generic point of the component $G$ of $\Ex(g)$ intersecting 
the proper transform $T$ of the $t$-axis $(s=0) \subset S$.

We consider the contraction $h \colon Z \to Z'$ of all the irreducible components
of $\Ex(g)$ that are contracted by $\f$. The map $\f$ factors through $h$ and a
morphism $\f' \colon Z' \to Y$. We look at the relative canonical divisor $K_{Z'/Y}$
of $\f'$, which we decompose as $K_{Z'/Y} = K_{Z'/Y}^\ver + K_{Z'/Y}^\hor$
by separating those components that exceptional over $S$ from those that are not. 
By a local computation in codimension one (like in H\"urwitz formula),
we check that 
\[
\ord_G(h^*K_{Z'/Y}^\ver) < \ord_G(\f^*E).
\]
On the other hand, a negativity lemma (essentially the Hodge index theorem) implies that
\[
\ord_G(h^*K_{Z'/Y}^\ver) \ge \ord_G(K_{Z/S}). 
\]
This is the step where we use the assumption that $Y$ is a minimal model.
To conclude, we just observe that $\ord_G(K_{Z/S}) \ge 1$ since $S$ is smooth, and
$\ord_G(\f^*E)=1$ since $T$ maps to an arc on $Y$ with order of contact one along $E$. 
This gives the contradiction we wanted.

\begin{rmk}
The computation at the core of 
the proof becomes particularly transparent if $Z' = Z$ (and $Y$ is smooth). 
In this special case, $K_{Z/Y}$ is the effective divisor
defined by the vanishing of the Jacobian of $\f$, and its coefficient at $G$ is zero 
because otherwise the special arc $\Phi(0)$ of the wedge could not meet
transversally $E$ in $Y$. The Jacobian of $\f$ gives a homomorphism 
from $\O_Z(\f^*K_Y)$ to $\O_Z(K_Z)$, and hence we have the linear 
equivalence $K_{Z/Y} \sim K_Z - \f^*K_Y$. Note, on the other hand, that
$K_Z$ is linear equivalent to the exceptional divisor $K_{Z/S}$, 
which is defined by the Jacobian of $g$. Combining these linear equivalences, we obtain the 
equivalence $K_{Z/S} - K_{Z/Y}^\ver \sim \f^*K_Y + K_{Z/Y}^\hor$.
The divisor in the right hand side is $g$-nef because $K_{Z/Y}^\hor$ is horizontal and
effective and $K_Y$ is $f$-nef, and the divisor in the left hand side is $g$-exceptional.
It follows by the negative definitness of the intersection form
of the exceptional divisors of $Z$ that $K_{Z/S} - K_{Z/Y}^\ver$ is anti-effective. 
This is impossible, however, since the coefficient of $G$ in $K_{Z/S}$ is positive 
and its coefficient in $K_{Z/Y}$ is zero.
\end{rmk}

\begin{rmk}
It would be interesting to try to extend the proof to  
positive characteristics. There are some difficulties to overcome. First of all, 
the proof of Theorem~\ref{t:main} relies on Reguera's curve selection lemma
which requires, in the construction of the wedge,
to take a field extension which may in principle be inseparable. 
This creates problems for the definition of $K_{Z'/Y}$.
Notice that this issue does not occur in dimension two, since
in this case we are led to work with wedges defined over finite algebraic extensions of $k$,
and if $k$ is algebraically closed then there is only the trivial one.
However, one runs into another difficulty: the local computation leading to the
first inequality displayed above breaks down 
if $\f'$ is wildly ramified. In order to extend the result of this paper
to positive characteristics, it seems that one would need to gain control on the
construction of the wedge to avoid these problems.
\end{rmk}

\subsection{Acknowledgments}
We are very grateful to Javier Fern\'andez de Bobadilla for many conversations
and for bringing to our attention an error in an earlier version of this paper. 
We thank Lawrence Ein and Shihoko Ishii for useful discussions, and the referees
for their careful reading of the paper and their relevant comments.

\section{Notation and conventions}

We work over an uncountable algebraically closed field $k$ of characteristic zero.
All schemes are separated and defined over $k$. 
A \emph{point} of a scheme is a schematic point; more generally, we will also consider
$A$-valued points on schemes, where $A$ is any $k$-algebra.
We denote by $\kappa_p$ the residue field of a point $p$. 
A point of an irreducible scheme $X$ is \emph{very general} if it can be
chosen arbitrarily in the complement of a countable union of proper closed subsets.
If $X$ is a variety, we denote by $X_\sm$ the smooth locus of $X$, and 
set $X_\sing = X \smallsetminus X_\sm$. 
The \emph{exceptional locus} of a proper birational morphism of schemes $f \colon Y \to X$ 
is the smallest set $\Ex(f) \subset Y$ away from which $f$ is 
an isomorphism to its image. A \emph{divisorial valuation} on a variety $X$
is a valuation of the form $\n = \ord_E$ where $E$ is a 
prime divisor on a resolution of singularities $Y \to X$.

\section{The Nash problem}

In this section, we quickly recall basic definitions and properties related to arc spaces
and the Nash map.
For a full treatment on arc spaces, we refer the reader to~\cite{DL99,EM09}.
For a more comprehensive introduction to the Nash problem, we refer to 
\cite{Nas95,IK03}.

\subsection{Arc space}

For any scheme $X$ over $k$ and any field extension $K/k$,
a \emph{$K$-valued arc} of $X$ is a morphism $\g \colon \Spec K[[t]] \to X$. 
We denote the image of the closed point by $\g(0)$ and 
the image of the generic point by $\g(\e)$. 

Let now $X$ be a scheme of finite type over $k$.   
For every $m \in \N$, the functor that to any $k$-algebra
$A$ associates $\Hom_{\text{$k$-Sch}}(\Spec A[t]/(t^{m+1}), X)$
is representable by a scheme of finite type over $k$, denoted $X_m$ and
called the \emph{$m$-th jet scheme} of $X$. For $m > n$, the natural
quotients $A[t]/(t^{m+1}) \to A[t]/(t^{n+1})$ induce morphisms $X_m \to X_n$, 
known as the \emph{truncation maps}. The truncation
maps are affine, and can be used to define a scheme $X_\infty = \liminv\,
X_m$, called the \emph{arc space} of $X$. 
The set of $K$-valued points of $X_\infty$ is 
\[
X_\infty(K) = \Hom_{\text{$k$-Sch}}(\Spec K[[t]], X),
\]
the set of $K$-valued arcs of $X$. 
The Zariski topology of $X_\infty$ coincides with the inverse limit topology.

The arc space comes equipped with natural projections $X_\infty \to X_m$. 
When $m=0$, this gives a morphism $\pi_X \colon X_\infty \to X$,
mapping an arc $\g \in X_\infty$ to $\g(0) \in X$.
If $\g$ is any $K$-valued arc, then we regard $\p_X(\g)$ as the $K$-valued point of $X$
given by pre-composition with the inclusion $\Spec K \inj \Spec K[[t]]$.

Given a morphism $f\colon Y \to X$ of $k$-schemes of finite type, we obtain an induced
morphism between the respective arc spaces $f_\infty\colon Y_\infty \to X_\infty$.
At the level of the functors of points, $f_\infty$ is just given by
composition with $f$, mapping a $K$-valued arc $\g \colon \Spec K[[t]] \to Y$
to the $K$-valued arc $f \o \g \colon \Spec K[[t]] \to X$.

Every $K$-valued arc $\a \colon \Spec K[[t]] \to X$ defines a valuation on the local ring
$\O_{X,\p_X(\a)}$, given by $\ord_\a(h) := \ord_t(\a^\sharp(h))$. The
valuation extends to a valuation of the function field of $X$ if and only
if $K((t))$ is a field extension of the function field of $X$, that is, if and only if
the image of $\a$ is dense in $X$.

\subsection{Nash valuations and the Nash map}

A divisorial valuation $\n$ is an \emph{essential valuation} over $X$ if 
the center of $\n$ on any resolution 
of $X$ is an irreducible component of the inverse image of $X_\sing$, 
and is a \emph{Nash valuation} over $X$ if it is the
valuation defined by the generic point of an irreducible component
of $\p_X^{-1}(X_\sing)$.

The decomposition of $\p_X^{-1}(X_\sing)$ into irreducible components, 
and the way these components define divisorial valuations over $X$,  
can be understood using resolutions of singularities.
This is well explained, for instance, in \cite{IK03}.
We briefly outline the argument. 

Let $f \colon Y \to X$ be any resolution of singularities, and
let 
\[
f^{-1}(X_\sing) = \bigcup_{i \in I} C_i
\]
be the decomposition of the inverse image of the singular locus into irreducible components. 
Notice that $I$ is a finite set. 
Let $Y_i \to Y$ be the normalized blow-up of $C_i$, and let $E_i$
be the prime divisor on $Y_i$ dominating $C_i$.

Since $Y$ is smooth, all truncation maps $Y_m \to Y_n$ are 
smooth, and this implies that each $\p_Y^{-1}(C_i) \subset Y_\infty$ is an irreducible set.
Then, denoting by $N_i$ the closure of $f_\infty(\p_Y^{-1}(C_i))$ in $X_\infty$, 
each $N_i$ is irreducible, and there is a unique minimal subset $I' \subset I$ such that
\[
\p_X^{-1}(X_\sing) = \bigcup_{i \in I'}N_i.
\]
For every $i \in I'$, the generic point $\a_i$ of $N_i$ is the image of the generic point
of $\p_{Y_i}^{-1}(E_i)$, and this implies that $\ord_{\a_i} = \ord_{E_i}$.
In particular, each valuation $\ord_{\a_i}$ is a divisorial valuation on $X$. 

Given the above resolution $f$, 
the center in $Y$ of any essential valuation over $X$ must be one of the irreducible components $C_i$. 
This means that there is a subset $I'' \subset I$ such that a 
divisorial valuation $\n$ over $X$ is essential if and only if $\n = \ord_{E_i}$
for some $i \in I''$. 
The decomposition of $\p_X^{-1}(X_\sing)$ into irreducible components
only depends on the topology of $X_\infty$, and not on the choice of resolution. 
This implies that $I' \subset I''$. In terms of valuations, 
it means that there is an inclusion
\[
\{\text{Nash valuations}\} \inj 
\{\text{essential valuations}\}.
\]
This is the \emph{Nash map}. 
The original formulation of the \emph{Nash problem} asked whether this map is surjective.

\subsection{Wedges}

Let $X$ be a scheme of finite type over $k$. For any field extension $K/k$,
a (parametrized) $K$-valued \emph{wedge} on $X$ is a morphism 
\[
\Phi \colon \Spec K[[s,t]] \to X.
\]

Since $K[[s,t]] = K[[s]][[t]]$, giving a $K$-valued wedge $\Phi \colon
\Spec K[[s,t]] \to X$ is equivalent to giving a $K[[s]]$-valued arc (which
we still denote by the same letter)
\[
\Phi \colon \Spec K[[s]] \to X_\infty.
\]
Restricting to the closed and generic points of $\Spec K[[s]]$, we obtain arcs
\[
\Phi_0 \colon \Spec K \to X_\infty
\and
\Phi_\e \colon \Spec K((s)) \to X_\infty.
\]
We can think of $\Phi_0$ as being a specialization of $\Phi_\e$. We call
$\Phi_0$ the \emph{special arc} of the wedge, and $\Phi_\e$ the
\emph{generic arc} of the wedge.

\begin{rmk}
It is important to note that in this discussion we are using the explicit
choice of coordinates $(s,t)$ on $\Spec K[[s,t]]$. When talking about
wedges, we implicitly assume that such a choice of coordinates has been
made. 
\end{rmk}

Let $f \colon Y \to X$ be a projective birational morphism of algebraic
varieties over $k$, and let
\[
\Phi \colon S = \Spec K[[s,t]] \to X 
\]
be a wedge. We say that $\Phi$ \emph{lifts} to $Y$ if it factors through $f$. 
Thinking of $\Phi$ as an arc on $X_\infty$, we also say that $\Phi$ lifts
to $Y_\infty$.

If both the special arc $\Phi_0$ and the generic arc $\Phi_\e$ are not
contained in the indeterminacies of $f^{-1}$, then, by the valuative criterion of
properness, both arcs lift (uniquely) to arcs on $Y$:
\[
\xymatrix{
& Y \ar[d]^f &&& Y \ar[d]^f \\
\Spec K[[t]] \ar[r]^(.65){\Phi_0} \ar[ur]^{\~\Phi_0} & X, && 
\Spec K((s))[[t]] \ar[r]^(.68){\Phi_\e} \ar[ur]^{\~\Phi_\e} & X.
}
\]
However, in general $\~\Phi_0$ will not be a specialization of $\~\Phi_\e$.
In fact, $\~\Phi_\e$, viewed as a $K((s))$-valued point on $Y_\infty$,
might fail to extend to an arc $\Spec K[[s]] \to Y_\infty$. This is related
to whether $\Phi$ lifts to $Y$ or not. The following result follows
immediately from the definitions.

\begin{lem}
\label{l:lift-aux}
With above terminology, the following are equivalent:
\begin{enumerate}
\item $\Phi$ lifts to $Y$,
\item the morphism $\~\Phi_\e \colon \Spec K((s))[[t]]
\to Y$ factors through $\Spec K[[s,t]]$,
\item the morphism $\~\Phi_\e \colon \Spec K((s)) \to
Y_\infty$ extends to an arc $\~\Phi_\e \colon \Spec K[[s]] \to Y_\infty$.
\end{enumerate}
\end{lem}

\subsection{The curve selection lemma}

The idea outlined in \S\ref{s:outline} to look at the Nash problem as a lifting problem for wedges 
was first considered in~\cite{LJ80}, where the notion
of wedges was introduced in the context of surface singularities.
This approach is not specific of dimension two, and provides a natural
way to address the Nash problem in all dimensions. 

Given two closed irreducible sets $M$ and $N$ in $X_\infty$, with $N \subsetneq M$,
the existence of a morphism $\Phi \colon \Spec K[[s]] \to X_\infty$ 
such that $\Phi(0) \in N$ and $\Phi(\e) \in M \smallsetminus N$
should be viewed as an algebro-geometric analogue of Milnor's curve selection lemma. 

In general, the curve selection lemma holds, in the algebro-geometric sense, for Noetherian schemes, 
where it can be proven by cutting down and using 
induction on dimension. However, $X_\infty$ is not a Noetherian
scheme, and there are examples where the curve selection lemma fails in the non-Noetherian setting
(e.g., see \cite[Example~4]{FdBPP}). 
It is therefore a delicate issue to
establish the existence of a wedge $\Phi$ with the above properties.

The first general result on the curve selection lemma for arc spaces is due to Reguera \cite{Reg06}. 
In the proof of Theorem~\ref{t:main} we will use
the following application of Reguera's result.
We first introduce some notation.

Let $f \colon Y \to X$ be a projective birational morphism from a normal variety $Y$, 
and let $E$ be a prime divisor on $Y$ contained in $f^{-1}(X_\sing)$. 
Let $\~\a \in Y_\infty$ be the generic point of the irreducible component
of $\p_Y^{-1}(E)$ dominating $E$, and let $\a = f_\infty(\~\a)$.
Note that $\a \in \p_X^{-1}(X_\sing)$. The closure of $\a$ is what was
denoted by $N_E$ in \S\ref{s:outline}. 
Recall that a point of a scheme is said to be very general if 
it is in the complement of a countable union of preassigned proper closed subsets. 

\begin{thm}
\label{t:csl}
With the above notation, suppose that $\a$ is not the generic point of an
irreducible component of $\p_X^{-1}(X_\sing)$. Then for a very general
point $p \in E$ with residue field $\kappa_p$, there is a finite algebraic
extension $K/\k_p$, and a wedge
\[
\Phi \colon  K[[s,t]] \to X
\]
that does not lift to $Y$, such that the lift $\~\Phi_0 \colon \Spec K[[t]]
\to Y$ of the special arc $\Phi_0$ has order of contact one with $E$ at
$p$. 
\end{thm}

This result follows from \cite[Corollary~4.6]{Reg06}
by a specialization argument along the lines of \cite[\S2.2]{LJR12} or \cite[\S3.2]{FdB12}. 
We explain the argument in details in the Appendix.

\section{Geometry on a resolution of a wedge}

Throughout this section, 
let $f \colon Y \to X$ be a projective birational morphism of varieties over $k$, and
consider a wedge
\[
\Phi \colon S = \Spec K[[s,t]] \to X
\]
where $K/k$ is a field extension.
We assume that $\Phi$ does not lift to $Y$, and
that the image of $\Phi$ in $X$ is not fully 
contained in the indeterminacy locus of $f^{-1}$.

\subsection{Resolution of the wedge}
\label{s:intersection}

We can identify $f$ with the blow-up of an ideal sheaf $\I \subset \O_X$ 
that does not vanish on the image of $\Phi$. 
Consider then the inverse image $\J = \Phi^{-1}\I \. \O_S$. We 
write $\J = \fa\.(h)$ where $h \in K[[s,t]]$ and $\fa$ is an $(s,t)$-primary
ideal. 

\begin{prop}
\label{p:h}
There is a commutative diagram
\[
\xymatrix{
Z \ar@/_1pc/[rdd]_g \ar[rd]^h\ar@/^1pc/[rrd]^\f &&\\
& Z' \ar[d]_{g'}\ar[r]^{\f'} & Y \ar[d]^f \\
& S \ar[r]^\Phi & X
}
\]
where $g'$ is the blow-up of the integral closure $\ov\fa$ of $\fa$, 
and $g$ is a minimal sequence of blow-ups of maximal ideals such that $h$ is a morphism. 
Both $Z$ and $Z'$ are two-dimensional and projective over $S$, $Z$ is regular, and $Z'$ is normal
and $\Q$-factorial. 
The exceptional loci $\Ex(g)$ and $\Ex(g')$ are non-empty sets of pure codimension one,
and their irreducible components are projective curves.
\end{prop}

\begin{proof}
The morphism $\f'$ exists by the universal property of the blow-up. 

Being an $(s,t)$-primary ideal, $\ov\fa$ is generated by polynomials, 
that is, $\ov\fa = \ov\fa^\o\.K[[s,t]]$
for an $(s,t)$-primary ideal $\ov\fa^\o \subset K[s,t]$. 
Note that $\ov\fa^\o = \ov\fa \cap K[s,t]$ is integrally closed. 
If $S^\o=\Spec K[s,t]$ and $(Z')^\o \to S^\o$ is the blow-up of $\ov\fa^\o$, then
we have that $Z'= (Z')^\o\times_{S^\o} S$. In particular, $Z'$ is normal and
projective over $S$. 

The indeterminacies of $S^\o \rat (Z')^\o$ can be resolved 
by a minimal sequence $Z^\o \to S^\o$ of blow-ups of maximal ideals. All centers of blow-up are
closed points, and thus they lie over the closed point of $S^\o$. 
We let $Z = Z^\o \times_{S^\o} S$. By base change, 
we obtain a morphism $g \colon Z \to S$, resolving the indeterminacies
of $S \rat Z'$, given by a composition of blow-ups of maximal ideals. 
We let $h$ and $\f$ be the induced morphisms.

To check that $Z'$ is $\Q$-factorial, we apply a result of Lipman. 
Since $S$ has rational singularities, we have $H^1(Z,\O_Z) = 0$
by \cite[Proposition~(1.2)]{Lip69}.
Using \cite[Lemma~(12.2)]{Lip69}, we see that this is 
implies that $H^1(h^{-1}(U),\O_Z) = 0$ for every open set $U \subset Z'$,
and therefore $Z'$ has rational singularities. 
Then \cite[Proposition~(17.1)]{Lip69} implies that $Z'$ is $\Q$-factorial.

By construction, all irreducible components of $\Ex(g)$ and $\Ex(g')$ are projective.
These sets are non-empty because $f^{-1}\o\Phi \colon S \rat Y$ is not a morphism.
\end{proof}

We say that an irreducible component $C$ of $\Ex(g)$ 
is \emph{contracted} by $\f$ if $\f$ maps every point of $C$
to the same point of $Y$. A similar definition is given for $\f'$ and the components of $\Ex(g')$. 

\begin{lem}
\label{l:contracted}
For an irreducible component $C$ of $\Ex(g)$, the following are equivalent:
\begin{enumerate}
\item
$C$ is a component of $\Ex(h)$;
\item
$C$ is contracted by $\f$; 
\item
there is a closed point $q \in C$ such that if 
$c \in C$ is the generic point, then $\f(q) = \f(c)$.
\end{enumerate}
\end{lem}

\begin{proof}
It is clear from the definition that $(a) \Rightarrow (b) \Rightarrow (c)$. 
Suppose then that $C$ is not a component of $\Ex(h)$.
Then $h(C)$ is a component of $\Ex(g')$. 
We can assume that $X$ is affine. 
The line bundle $\cL = f^{-1}\I\.\O_Y$ on $Y$ is relatively very ample over $X$.
By construction, $Z'$ is the normalization of the blow-up of 
$\Phi^{-1}\I\.\O_S$, and therefore the line bundle $(\f')^*\cL$ is relatively ample over $S$.
Note that $(\f')^*\cL$ is globally generated by sections that are pulled back from 
sections of $\cL$. 
In particular, for every closed point $q \in C$ we can find a section $s \in H^0(Y,\cL)$ 
such that $(\f^*s)(q)=0$ and $(\f^*s)(c) \ne 0$. This means that $s(\f(p))=0$
and $s(\f(c)) \ne 0$, and hence $\f(q) \ne \f(p)$.
We conclude that $(c) \Rightarrow (a)$.
\end{proof}

Since $Z$ is smooth, every divisor on $Z$ is Cartier.
Note that every divisor on $Z$ is linearly equivalent to
a unique $g$-exceptional divisor.
The intersection product of a divisor $D$ with an 
irreducible component $C$ of $\Ex(g)$, is defined by
$D\.C := \deg(\O_Z(D)|_C)$ (cf.\ \cite[\S10]{Lip69}).
This product extends to all $\Q$-divisors on $Z$. 
A $\Q$-divisor $D$ is $g$-nef (resp., $h$-nef)
if $D\.C \ge 0$ for every irreducible component $C$ of $\Ex(g)$ (resp., of $\Ex(h)$).
Note that if $D$ is effective and its support does not contain any irreducible component
of $\Ex(g)$, then $D$ is $g$-nef.

If $C_1,\dots,C_r$ are the irreducible components of $\Ex(g)$, 
then the intersection matrix $(C_i\.C_j)$ is negative definite
(e.g., see \cite[Lemma~(14.1)]{Lip69}).
One deduces the following negativity lemma (cf.\ \cite[Lemma~3.41]{KM98}).

\begin{lem}
\label{l:negativity}
For any subset $I \subset \{1,\dots,r\}$, let
$D = \sum_{i\in I} a_i C_i$ be a $\Q$-divisor such that $D\.C_i \ge 0$ for every $i \in I$. 
Then $a_i \le 0$ for all $i$. 
\end{lem}

Let $L$ be a $\Q$-Cartier $\Q$-divisor on $Y$, and assume that $\f(Z)$ is not contained
in the support of $L$. Let $r$ be a positive integer such that $rL$ is Cartier. 
By pulling back the local equations of $rL$ to $Z$ and dividing by $r$, 
we define the pull-back $\f^*L$, which is a $\Q$-divisor on $Z$. 

\begin{lem}
\label{l:nef}
With the above notation, if $L$ is $f$-nef then $\f^*L$ is $g$-nef.
\end{lem}

\begin{proof}
We can assume without loss of generality that $X$ is affine. 
Let $C$ be any irreducible component of $\Ex(g)$.
By the continuity of the degree function on $\Q$-divisors on $C$, it suffices to 
prove that $\f^*L\.C \ge 0$ when $L$ is ample. But this is clear, since 
in this case $\O_Y(mL)$ is globally generated for some $m \ge 1$, and thus we can write
$L \sim \frac 1m H$ where $H$ is an effective divisor that does not contain the image of $C$. 
\end{proof}

Since $Z'$ is $\Q$-factorial, we can define the intersection product $D\.C$
between a divisor $D$ on $Z'$ and an irreducible component $C$
of $\Ex(g')$. As usual, we say that $D$ is $g'$-nef if $D\.C=0$ for every $C$.
The pull-back $h^*D$ is also defined, by rescaling. 
Notice that $D$ is $g'$-nef if and only if $h^*D$ is $g$-nef.

\subsection{Canonical divisors}
\label{s:can-div}

The next step is to define a canonical divisor of $Z$. This is not
straightforward, as $Z$ is not a variety over a field. 
Since $Z$ is a scheme of finite type over $K[[s,t]]$,
we need to work with \emph{special differentials}. 
We follow the treatment given in~\cite[Appendix~A]{dFEM11}.

Let $R = F[[x_1,\dots,x_n]]$ where $F$ is a field of characteristic zero.
For any $R$-module $M$, 
an $F$-derivation $D \colon R \to M$
is said to be a \emph{special $F$-derivation} if
$D(h) = \sum_i \frac{\de h}{\de x_i}D(x_i)$
for every $h \in R$.
If $A$ is an $R$-algebra and $M$ is an $A$-module, then the 
module ${\rm Der}'_F(A,M)$ of \emph{special $F$-derivations} consists of all derivations
$D \colon A \to M$ that restrict to special $F$-derivations on $R$.
For any $R$-algebra $A$ there is an $A$-module $\Om_{A/F}'$ and
a special $F$-derivation $d'_{A/F} \colon A \to \Om_{A/F}'$ that induces an isomorphism of $A$-modules
\[
\Hom_A(\Om_{A/F}',M) \to {\rm Der}'_F(A,M)
\] 
for every $A$-module $M$ \cite[Corollary~A.4]{dFEM11}.
The module $\Om_{A/F}'$ is called the \emph{module of special differentials}
of $A$ over $F$. 
As for usual differentials, special differentials commute with localization
\cite[Lemma~A.7]{dFEM11}. 
It follows then by the second part of \cite[Corollary~A.4]{dFEM11}
that if $A$ is essentially of finite type over $R$, then 
$\Om_{A/F}'$ is a finitely generated $A$-module. 

For any scheme $T$ over $R$, we obtain a coherent sheaf
$\Om_{T/F}'$, called the sheaf of \emph{special differentials} of $T$ over $F$.
For example, if $T = \Spec R$, then $\Om_{T/F}'$ is the
free $\O_T$-module generated by $dx_1,\dots,dx_n$ \cite[Lemma~A.2]{dFEM11}. 
If $T$ is essentially of finite type over $R$, then $\Om_{T/F}'$ is a coherent sheaf. 
If furthermore $T$ is smooth of pure dimension $a$, and the residue field $\kappa_p$ of 
a closed point $p \in T$ has transcendence degree $b$ over $F$, then 
it follows from \cite[Proposition~A.8]{dFEM11} that $\Om_{T/F}'$ is a locally free sheaf of rank $a+b$. 

Suppose now that $T$ is a reduced scheme essentially of finite type over $R$, 
and that $F$ is a finitely generated extension of a subfield $E$. 
Note that $T$ is a scheme essentially of finite type over $E[[x_1,\dots,x_n]]$, too. 
If $d$ is the transcendence degree of $F/E$, then
$\Om_{F/E}$ is an $F$-vector space of dimension $d$. 

\begin{prop}
\label{p:Om_F/E}
With the above assumptions, there is a short exact sequence
\[
0 \to \Om_{F/E}\otimes \O_T \to \Om_{T/E}' \to \Om_{T/F}' \to 0.
\]
\end{prop}

\begin{proof}
For every affine chart $\Spec A \subset T$, and every $A$-module $M$, 
we have an exact sequence
\[
0 \to {\rm Der}'_F(A,M) \to {\rm Der}'_E(A,M) \to {\rm Der}_E(F,M) \otimes A.
\]
This implies that the sequence in the statement is exact in the middle and on the right. 
The exactness on the left follows by observing that 
$\Om_{F/E}\otimes \O_T \cong \O_T^{\otimes d}$,
and the difference between the ranks of $\Om_{T/E}'$ and $\Om_{T/F}'$ is equal to $d$.
\end{proof}

We now come back to our setting. 
Since $Z$ is regular of dimension two and the residue fields of its closed points 
are finite extensions of $K$, $\Om_{Z/K}'$ is a locally free sheaf, of rank two.
Then there is a divisor $K_Z$ on $Z$ such that
\[
\O_{Z}(K_Z) \cong \wedge^2 \Om_{Z/K}'.
\]
We say that $K_Z$ is a \emph{canonical divisor of $Z$}. 

The natural map of invertible sheaves
\[
\wedge^2\Om_{S/K}'\otimes\O_Z \to \wedge^2\Om_{Z/K}'
\]
is defined by multiplication of a local equation of an effective divisor
supported on the exceptional locus $\Ex(g)$ \cite[Lemma~A.11]{dFEM11}. We denote this divisor by
$K_{Z/S}$ and call it the \emph{relative canonical divisor} of $Z$ over $S$.
Note that, since $\wedge^2\Om_{S/K}' \cong \O_S$, we have
\begin{equation}
\label{eq:K_Z/S}
K_Z \sim K_{Z/S}.
\end{equation}

If $K/k$ is a finitely generated field extension, then 
$Z$ can also be regarded as a scheme essentially of finite type over $k[[s,t]]$.
The sheaf $\Om_{Z/k}'$ of special differentials of $Z$ over $k$
is a locally free sheaf of rank $d+2$, 
where $d$ is the transcendence degree of $K/k$.
It follows by Proposition~\ref{p:Om_F/E} that
\[
\O_{Z}(K_Z) \cong \wedge^{d+2} \Om_{Z/k}'.
\]
In particular, the definition of canonical divisor on $Z$ is
independent of whether we take (special) differentials over $K$ or over $k$. 

\begin{rmk}
\label{r:Om_Z^o}
With the notation as in the proof of Proposition~\ref{p:h}, 
it follows by~\cite[Proposition~A.10]{dFEM11} that 
$\Om_{Z^\o/K}\otimes \O_Z \cong \Om_{Z/K}'$
(and, similarly, $\Om_{Z^\o/k}\otimes\O_Z \cong \Om_{Z/k}'$).
In particular, $K_{Z/S}$ is the pull-back of $K_{Z^\o/S^\o}$, and
$K_Z$ is linearly equivalent to the pull-back of $K_{Z^\o}$.
\end{rmk}

\subsection{Relative canonical divisors}
\label{s:ramif}

For the reminder of this section, we let $n = \dim Y$.
We assume henceforth that $K/k$ has transcendence degree $d=n-2$,
and $\f$ is dominant. 

Let $a \colon \~Y \to Y$ be a projective resolution of singularities.
By taking a sequence of blow-ups $b \colon \~Z \to Z$ centered at closed points, 
we obtain a commutative diagram
\[
\xymatrix{
\~Z \ar[d]_{b} \ar[r]^{\~\f} & \~Y \ar[d]^a \\
Z \ar[r]^{\f} & Y
}
\]
where $\~Z$ is smooth. Note that $\~\f$ is dominant. 

Let $K_{\~Z}$ be a canonical divisor of $\~Z$, defined
by the condition $\O_{\~Z}(K_{\~Z}) \cong \wedge^n \Om'_{\~Z/k}$. 
It is important here to keep in mind that we work with divisors, 
and not divisor classes. 
By \cite[Lemma~A.11]{dFEM11}, we can choose $K_{\~Z}$ such that $b_*K_{\~Z} = K_Z$. 
We fix a canonical divisor $K_{\~Y}$. 

Consider the natural map
\[
\a \colon \Om_{\~Y/k} \otimes \O_{\~Z} \to \Om_{\~Z/k}'.
\]
Note that this is a map of locally free sheaves of rank $n$. 

\begin{lem}
\label{l:ramif}
The map of invertible sheaves
\[
\wedge^n \a \colon \wedge^n\Om_{\~Y/k} \otimes \O_{\~Z} \to \wedge^n\Om_{\~Z/k}'
\]
is locally given by multiplication of an equation of an effective divisor $K_{\~Z/\~Y}$
linearly equivalent to $K_{\~Z} - \~\f^*K_{\~Y}$. 
\end{lem}

\begin{proof}
Let $y \in \~Y$ and $z \in \~Z$ be the generic points, 
and denote by $\k_y$ and $\k_z$ their residue fields.
Note that $\~\f(z) = y$, since $\~\f$ is dominant.
Localizing at the generic points, we see that 
$(\wedge^n \Om_{\~Y/k})_y = \k_y$ and $(\wedge^n \Om_{\~Z/k}')_z = \k_z$,
and the induced map $(\wedge^n \Om_{\~Y/k})_y \otimes \k_z \to (\wedge^n \Om_{\~Z/k}')_z$
is an isomorphism. This implies that 
$\wedge^n\a$ is injective, and the assertion follows. 
\end{proof}

Suppose now that $Y$ is a normal variety 
with a $\Q$-Cartier canonical divisor $K_Y$.
We can take $K_Y = a_*K_{\~Y}$. 
Let $K_{\~Y/Y} := K_{\~Y} - a^*K_Y$ be the relative canonical divisor of $a \colon \~Y \to Y$.
Then we define the \emph{relative canonical divisor} of $\f$ to be
\[
K_{Z/Y} := b_*\big(K_{\~Z/\~Y} + \~\f^*K_{\~Y/Y}\big).
\]

\begin{lem}
The definition of $K_{Z/Y}$ is independent of the choice of the models $\~Y$ and $\~Z$. 
\end{lem}

\begin{proof}
It suffices to check that the definition does not change if we replace $\~Y$ and $\~Z$
by some other models $\^Y$ and $\^Z$ dominating them. Let
$\b \colon \^Z \to \~Z$ and $\^\f \colon \^Z \to \^Y$ 
denote the corresponding morphisms.
Applying several times the chain rule, we see that
\[
K_{\^Z/\^Y} + \^\f^*K_{\^Y/Y} = K_{\^Z/\~Z} + \b^*\big(K_{\~Z/\~Y} + \~\f^*K_{\~Y/Y}\big).
\]
The push-forward of this divisor to $\~Z$ agrees with $K_{\~Z/\~Y} + \~\f^*K_{\~Y/Y}$
by the projection formula and the fact that $K_{\^Z/\~Z}$ is $\b$-exceptional. 
\end{proof}

\begin{prop}
\label{p:ramif}
We have
\[
K_{Z/Y} \sim K_Z - \f^*K_Y.
\]
Moreover, if $Y$ has canonical singularities, then $K_{Z/Y} \ge 0$.
\end{prop}

\begin{proof}
We have
\[
K_{\~Z/\~Y} + \~\f^*K_{\~Y/Y} \sim K_{\~Z} - \~\f^*(a^*K_Y)
\]
by Lemma~\ref{l:ramif}.
Note that $b_*K_{\~Z} = K_Z$ and $b_*(\~\f^*(a^*K_Y)) = \f^*K_Y$. 
Then the first assertion is a consequence of the fact that pushing forward
preserves linear equivalence. 
Regarding the last assertion, it suffices to recall that 
$K_{\~Z/\~Y} \ge 0$, and observe that
$K_{\~Y/Y} \ge 0$ (and hence $\~\f^*K_{\~Y/Y}\ge 0$)
if $Y$ has canonical singularities.
\end{proof}

\begin{prop}
\label{p:hurwitz}
Let $E$ be a prime $\Q$-Cartier divisor on $Y$, 
and suppose that $C$ is a prime divisor on $Z$ dominating $E$.
Then 
\[
\ord_C(K_{Z/Y}) = \ord_C(\f^*E)-1.
\]
\end{prop}

\begin{proof}
Let $p \in Y$ and $q \in Z$ be, respectively, the generic points of $E$ and $C$, 
and let $\k_p$ and $\k_q$ be their residue fields.
By Cohen structure theorem, the completed local rings 
$\^{\O_{Y,p}}$ and $\^{\O_{Z,q}}$ have coefficient fields.
We consider the morphism 
\[
\ff\colon W = \Spec\^{\O_{Z,q}} \longrightarrow V = \Spec\^{\O_{Y,p}}
\]
induced by $\f$. 
By localizing at $p$ and applying \cite[Proposition~A.10]{dFEM11}, 
we see that $\Om_{Y/k} \otimes \O_V \cong \Om'_{V/k}$.
Similarly, recalling that $\Om_{Z/k}' \cong \Om_{Z^\o/k} \otimes \O_Z$
(cf.\ Remark~\ref{r:Om_Z^o}), 
we see that $\Om'_{Z/k} \otimes \O_W \cong \Om'_{W/k}$
by applying the same argument to the map $W \to Z^\o$.
Therefore the local equation of $K_{Z/Y}$ in $\O_W$ 
is given by the determinant of the Jacobian matrix
defining the map $\Om_{V/k}' \otimes \O_W \to \Om_{W/k}'$.
We have the commutative diagram
\[
\xymatrix{
0 \ar[r] & \Om_{\k_p/k}\otimes \O_W \ar@{=}[d] \ar[r] 
& \Om_{V/k}'\otimes\O_W \ar[d] \ar[r] & \Om_{V/\k_p}'\otimes\O_W \ar[r]\ar[d] & 0 \\
0 \ar[r] & \Om_{\k_p/k}\otimes \O_W \ar[r] & \Om_{W/k}' \ar[r] & \Om_{W/\k_p}' \ar[r] & 0,
}
\]
where the rows are exact by Proposition~\ref{p:Om_F/E} and the fact
that $\Om_{V/\k_p}'$ is locally free. 
From the diagram, we see that $K_{Z/Y}$ is locally defined by the equation 
of the map $\Om_{V/\k_p}' \otimes \O_W \to \Om_{W/\k_p}'$.
Since the extension $\k_q/\k_p$ is finite, we have $\Om_{\k_q/\k_p} \otimes \O_W \cong \O_W$, 
and therefore $\Om_{W/\k_p}' \cong \Om_{W/\k_q}'$ by Proposition~\ref{p:Om_F/E}.
Then the assertion follows by the following standard computation.
We fix isomorphisms $\O_V \cong \k_p[[v]]$ and $\O_W \cong \k_q[[w]]$,  
so that $\ff$ is given by the equation $v = uw^a$ where $u$ is a unit in $\O_W$ and
$a = \ord_C(\f^*E)$. 
Then $\Om_{V/\k_p}'$ is generated by 
$dv$, $\Om_{W/\k_q}'$ is generated by $dw$, and we have $dv = uaw^{a-1}dw + w^adu$, 
which shows that $\ord_C(K_{Z/Y})=a-1$.
\end{proof}

Still in the setting of Proposition~\ref{p:h}, we define 
the canonical divisor $K_{Z'}$ of $Z'$ and the relative canonical divisor $K_{Z'/Y}$ of $\f'$
to be, respectively, $h_*K_Z$ and $h_*K_{Z/Y}$. 

The next corollary follows immediately from Proposition~\ref{p:ramif}. 

\begin{cor}
\label{c:ramif}
We have $K_{Z'/Y} \sim K_{Z'} - \f'^*K_Y$. 
Moreover, if $Y$ has canonical singularities, then $K_{Z'/Y} \ge 0$.
\end{cor}

The \emph{relative canonical divisor} of $h$ is defined by
\[
K_{Z/Z'} := K_Z - h^*K_{Z'}.
\]
Note that since $K_{Z'} = h_*K_Z$, $K_{Z/Z'}$ is $h$-exceptional. 

\begin{prop}
\label{p:K_Z/Z'}
We have $K_{Z/Z'} \le 0$. 
\end{prop}

\begin{proof}
By Lemma~\ref{l:negativity}, it suffices to show that $K_{Z/Z'}$ is $h$-nef.
This property is well-known to hold, under similar assumptions, for surfaces
of finite type over an algebraically closed field. We reduce to that case as follows
(we keep the notation as in the proof of Proposition~\ref{p:h}).

Let $\ov K$ be the algebraic closure of $K$, and let $G$ be the Galois group.  
The surface $Z^\o_{\ov K}$ is obtained from $S^\o_{\ov K}$ by a minimal sequence of blow-ups
of $G$-orbits of smooth closed points such that the induced rational map
$h^\o_{\ov K} \colon Z^\o_{\ov K}\rat (Z')^\o_{\ov K}$ is a morphism. 
Any irreducible component $C$ of $\Ex(g)$ is the pull-back of 
a curve $C^\o$ to $Z^\o$, whose pull-back to $Z^\o_{\ov K}$ is a union of 
finitely many disjoint curves $D_1,\dots,D_m$ forming an orbit under 
the action of $G$. Moreover, $K_{Z^\o_{\ov K}}$ and $K_Z$ are the pull-backs
(under the respective maps) of $K_{Z^\o}$ (cf.\ Remark~\ref{r:Om_Z^o}).
If $C$ is $h$-exceptional, then each $D_i$ is $(h^\o_{\ov K})$-exceptional.
Therefore we have $K_{Z^\o_{\ov K}} \. D_i \ge 0$
by the minimality of the sequence of $G$-equivariant blow-ups, 
since otherwise the adjunction formula would imply that $D_i$ is a $(-1)$-curve.  
This implies that $K_Z\.C \ge 0$, and hence $K_{Z/Z'}\.C \ge 0$. 
\end{proof}

\section{Proof of Theorem~\ref{t:main}}

Let $X$ be a variety over $k$. 
Theorem~\ref{t:main} is well-known to hold in dimension one (cf.\ Remark~\ref{r:dim-1}).
We can then assume that $\dim X \ge 2$.

Suppose that $\n$ is a terminal valuation over $X$, as defined in the introduction.  
This means that there is a minimal model $f \colon Y \to X$, and a prime exceptional divisor $E$
on $Y$, such that $\n = \ord_E$. Note that $E \subset f^{-1}(X_\sing)$.
We can assume without loss of generality that $Y$ is $\Q$-factorial
(cf.\ \S\ref{s:Q-fact}).
Let $\~\a \in Y_\infty$ be the generic point of the irreducible 
component of $\p_Y^{-1}(E)$ dominating $E$, 
and let $\a = f_\infty(\~\a) \in X_\infty$. Then $\a \in \p_X^{-1}(X_\sing)$
and $\n = \ord_\a$.  

The valuation $\n$ is a Nash valuation over $X$ if and only if $\a$ is 
the generic point of an irreducible component of $\p_X^{-1}(X_\sing)$.
We suppose by way of contradiction that this is not the case. 
Then we are in the setting of Theorem~\ref{t:csl}.

Let $p \in E$ be a very general point of codimension one. 
By Theorem~\ref{t:csl}, there is a finite algebraic extension $K/\k_p$, and a wedge
\[
\Phi \colon  \Spec K[[s,t]] \to X
\]
that does not lift to $Y$, such that the lift
$\~\Phi_0$ of the special arc $\Phi_0$ 
is an arc on $Y$ with order of contact one with $E$ at $p$.

Let
\[
\xymatrix{
Z \ar@/_1pc/[rdd]_g \ar[rd]^h\ar@/^1pc/[rrd]^\f &&\\
& Z' \ar[d]_{g'}\ar[r]^{\f'} & Y \ar[d]^f \\
& S \ar[r]^\Phi & X 
}
\]
be the diagram given in Proposition~\ref{p:h}.
Since the wedge $\Phi$ does not lift to $Y$, $g'$ is not an isomorphism, 
and therefore $\Ex(g') \ne \emptyset$.
Let $G$ be the irreducible component of $\Ex(g)$ intersecting the
proper transform $T$ of the $t$-axis $(s=0) \subset S$.

\begin{lem}
\label{l:C-dominates-E}
Every irreducible component of $\Ex(g')$ containing $h(G)$ dominates $E$.
\end{lem}

\begin{proof}
By construction, $p \in \f(G) \subset E$.
Recall that $p$ has codimension one in $E$, and is not contained in any other 
irreducible component of $\Ex(f)$. 
Let $C$ be any irreducible component of $\Ex(g')$ containing 
$h(G)$. Note that $p \in \f'(C)$. Since $C$ is irreducible and is contracted by $f\o\f'$, 
and $p$ is a very general point of $E$, we have $\f'(C) \subset E$. 
On the other hand, Lemma~\ref{l:contracted} implies that
$\f'(C) \ne p$. Therefore, as $p$ has codimension one in $E$, $C$ must dominate $E$.
\end{proof}

\begin{lem}
\label{l:dominant}
The morphism $\f$ is dominant.
\end{lem}

\begin{proof}
Let $\ov{\f(Z)} \subset Y$ be the Zariski closure of $\f(Z)$. 
Note that $\ov{\f(Z)}$ is irreducible, since $Z$ is irreducible.  
By Lemma~\ref{l:C-dominates-E}, we see that $E \subset \ov{\f'(Z')} = \ov{\f(Z)}$. 
On the other hand, 
since the lift $\~\Phi_0$ of the special arc has finite order of contact with $E$ at $p$, we 
have $\ov{\f(Z)} \not\subset E$. As $E$ has codimension one in $Y$, 
we conclude that $\ov{\f(Z)} = Y$. 
\end{proof}

Recall that, by construction, $K$ has transcendence degree $n-2$ over $k$. 
Then we are in the setting of \S\ref{s:ramif}, 
and so the relative canonical divisors $K_{Z/Y}$ and
$K_{Z'/Y}$ are defined.

The completion of the proof will result from a comparison between
the coefficient of $G$ in the relative canonical divisor $K_{Z/S}$, and its coefficient in $\f^*E$. 
The relative canonical divisors of $\f$ and $\f'$ will be used to link these two coefficients.

We start from $K_{Z/S}$.
Since $S$ is smooth and $G$ is $g$-exceptional, we have
\begin{equation}
\label{eq:0}
\ord_G(K_{Z/S}) \ge 1.
\end{equation}

Recall that $K_{Z/S} = K_{Z/Z'} + h^*K_{Z'/S}$. Since $K_{Z/Z'} \le 0$
by Proposition~\ref{p:K_Z/Z'}, we have
\begin{equation}
\label{eq:2}
K_{Z/S} \le h^*K_{Z'/S}.
\end{equation}

By Proposition~\ref{p:ramif} (see also Corollary~\ref{c:ramif})
and the fact that $K_S \sim 0$, we have
\[
\f'^*K_Y \sim K_{Z'} - K_{Z'/Y} \sim K_{Z'/S} - K_{Z'/Y}. 
\]
We decompose $K_{Z'/Y} = K_{Z'/Y}^\ver + K_{Z'/Y}^\hor$
where every component of $K_{Z'/Y}^\ver$ is $g'$-exceptional, and none of 
the components of $K_{Z'/Y}^\hor$ is. 
As $Y$ has terminal singularities, we have $K_{Z/Y} \ge 0$ by 
Proposition~\ref{p:ramif}, and since $K_{Z'/Y} = h_*K_{Z/Y}$, 
this implies that $K_{Z'/Y}^\hor \ge 0$.
Therefore $K_{Z'/Y}^\hor$, being effective and not containing any $g'$-exceptional curve, 
is $g'$-nef. Note that $\f'^*K_Y$ is also $g'$-nef, because we are assuming that $K_Y$ is $f$-nef.
It follows that $\f'^*K_Y + K_{Z'/Y}^\hor$ is $g'$-nef. 
Notice that 
\[
\f'^*K_Y + K_{Z'/Y}^\hor \sim K_{Z'/S} - K_{Z'/Y}^\ver,
\]
and the $\Q$-divisor on the right hand side is $g'$-exceptional.
Then Lemma~\ref{l:negativity} implies that
\begin{equation}
\label{eq:3}
K_{Z'/S} \le K_{Z'/Y}^\ver.
\end{equation}
 
Since $g'$ is the blow-up of an ideal cosupported at the closed point of $S$, 
the fiber of $g'$ over this point has pure dimension one, 
and its support is equal to $\Ex(g')$. As $G$ lies over the closed point of $S$, 
it follows that $h(G) \in \Ex(g')$.  
Let $C$ be any irreducible component of $\Ex(g')$ containing $h(G)$. 
By Lemma~\ref{l:C-dominates-E}, $C$ dominates $E$.
Then, by Proposition~\ref{p:hurwitz}, we have
\[
\ord_C(K_{Z'/Y}) = \ord_C(\f'^*E) - 1
\]
(this computation 
can be carried out on $Z$). Note that there exists at least one such component $C$ 
because $g'$ is not an isomorphism (if $G$ is not $h$-exceptional, then $C = h(G)$). 
Since $K_{Z'/Y}^\ver$ is supported on $\Ex(g')$ and $\f'^*E$ is effective,
the previous formula implies the
divisor $\f'^*E - K_{Z'/Y}^\ver$ is effective and nontrival in a neighborhood of the
generic point of $h(G)$. Since $Z'$ is $\Q$-factorial, it follows
that $\f^*E - h^*K_{Z'/Y}^\ver = h^*(\f'^*E - K_{Z'/Y}^\ver)$ is an effective and nontrivial
$\Q$-divisor in a neighborhood of the generic point of $G$. This means that
\begin{equation}
\label{eq:4}
\ord_G(h^*K_{Z'/Y}^\ver) < \ord_G(\f^*E).
\end{equation}

Finally, notice that the special arc $\~\Phi_0 \colon \Spec K[[t]] \to Y$
factors through a morphism $\ff \colon \Spec K[[t]] \to Z$
which gives a parameterization of $T$. 
Recall that $\~\Phi_0$ has order of contact one with $E$.
This means that $\ord_t(\~\Phi_0^*E) = 1$. Since 
$\ord_t(\~\Phi_0^*E) \ge \ord_t(\ff^*G)\.\ord_G(\f^*E)$, 
we must have
\begin{equation}
\label{eq:1}
\ord_G(\f^*E) = 1.
\end{equation}

Putting \eqref{eq:0}--\eqref{eq:1} together, we get
\[
1 
\le \ord_G(K_{Z/S}) \le \ord_G(h^*K_{Z'/S}) \le \ord_G(h^*K_{Z'/Y}^\ver) < \ord_G(\f^*E) = 1,
\]
which gives the desired contradiction. This completes the proof of Theorem~\ref{t:main}.

\section{Terminal valuations over toric varieties}
\label{s:toric}

Here we give a combinatorial description of terminal valuations in the case of toric varieties.

\subsection{Minimal models}
\label{s:Q-fact}

In dimension $\ge 3$, minimal models are not unique. 
Given a variety $X$, any minimal model program over $X$, originating from a
projective resolution of singularities of $X$, terminates with a minimal model over $X$
\cite{BCHM10}. 
Minimal models obtained in this way are $\Q$-factorial  
and are all isomorphic in codimension one.
Every other minimal model is dominated by a $\Q$-factorial minimal model
via a small birational morphism. 

It follows that the set of terminal valuations over $X$ is equal to the set of valuations
$\ord_{E_i}$ where $E_1,\dots, E_m$ are the prime exceptional divisors
on any given minimal model over $X$.
In a way, minimal models play the role of minimal resolutions of surfaces.

\subsection{Toric varieties}

With regards to toric varieties, we use the notation and terminology of
\cite{Ful93}. We fix an algebraic torus $T$, and denote by $M$ the
character lattice of $T$. We let $N = \Hom(M,\Z)$ be the dual lattice, and
denote $M_\R = M \otimes \R$ and $N_\R = N \otimes \R$.
The toric variety corresponding to a fan $\D$ in $N$ is denoted by $X(\D)$.
We denote by $\D(i)$ the set of $i$-dimensional cones in $\D$, 
and by $\D(1)_\prim$ the set of primitive elements in the intersection of the rays of $\D$
with the lattice $N$. Note that the elements in $\D(1)_\prim$ correspond to the
prime $T$-divisors on $X(\D)$. 
If $\s \subset N_\R$ is a strongly convex rational polyhedral cone, 
then we define, in a similar fashion, $X(\s)$, $\s(i)$, and $\s(1)_\prim$
by identifying, in the notation, $\s$ with the fan defined by its faces. 

We fix a strongly convex rational polyhedral cone $\s$, and consider the toric
variety $X(\s)$. 
The elements of $\s \cap N$ are in bijection with the torus-invariant
valuations on $X(\s)$. 
As we shall see below, terminal valuations, Nash valuations, 
and essential valuations over $X(\s)$
are all torus-invariant, and can be characterized in terms of
lattice properties of $\s$. 

A cone $\t$ is said to be regular if the vectors in $\t(1)_\prim$ form part of a
basis of $N$. This is equivalent to $X(\t)$ being smooth. A cone that is
not regular is called singular. We write 
$\s_{\rm sing} = \bigcup_\t \t^\circ$, where $\t$ ranges among the singular
faces of $\s$, and $\t^\circ$ denotes the relative interior of $\t$. Then
$\s_{\rm sing} \cap N$ parameterizes torus-invariant valuations on $X(\s)$
whose center is contained in the singular locus $X(\s)_\sing$. Terminal
valuations, Nash valuations, and
essential valuations all belong to $\s_\sing\cap N$.

Given two vectors $v, v' \in \s$, we write $v \leq_\s v'$ if $v' \in v+\s$.
It is clear that $\leq_\s$ is a partial order in $\s$. In~\cite{IK03}, the authors
prove that Nash and essential valuations coincide, and characterize them in
terms of the order $\leq_\s$. To state their result, let
\[
\Min(\s) := \{\text{minimal elements of $\s_\sing \cap N$ with respect to $\leq_\s$}\}.
\]

\begin{thm}[{\cite{IK03}}]
With the above notation, 
\[
\Min(\s) = \{\text{Nash valuations over $X(\s)$}\} =
\{\text{essential valuations over $X(\s)$}\}. 
\]
\end{thm}

We denote by
$\G(\s) \subset N_\R$ the convex hull of the non-zero elements in $\s
\cap N$, and let $\de_c\G(\s)$ be the union of the compact faces of $\G(\s)$.
Then let
\[
\Ter(\s) := \de_c\G(\s) \cap \s_\sing \cap N.
\]

\begin{prop}
With the above notation, 
\[
\Ter(\s) = \{\text{terminal valuations over $X(\s)$}\}.
\]
\end{prop}

\begin{proof}
Let $\D$ be a simplicial subdivision of $\s$ constructed as follows. 
For every maximal dimensional face $F$ of $\de_c\G(\s)$, we choose a
triangulation with set of vertices equal to $F \cap N$. 
Then we take $\D$ to be the cone over the resulting triangulation of $\de_c\G(\s)$. 
Since $\D$ is simplicial, the corresponding toric variety $X(\D)$ is $\Q$-factorial. 

Note that 
\[
\D(1)_\prim = \de_c\G(\s) \cap N.
\]
Let $v_1,\dots,v_m$ be the elements in this set. 
Each vector $v_i$ corresponds to a prime $T$-divisor $D_i$ on $X(\D)$, 
and $X(\D)$ has canonical class
$K_{X(\D)} \sim - \sum_iD_i$.
If $n = \dim X(\s)$, then the elements in $\D(n-1) \smallsetminus \s(n-1)$
are in one-to-one correspondence with toric invariant curves that are 
exceptional over $X(\s)$.
These curves generate the space of numerical classes of 1-cycles of $X(\D)$ over $X(\s)$. 
By the convexity of $\G(\s)$, it follows by a standard computation that
$K_{X(\D)}\.\g \ge 0$
for any such curve $\g$ (e.g., see the first exercise discussed in \cite[Page~99]{Ful93}).
This implies that $X(\D)$ is a minimal model over $X(\s)$. 

Since $X(\D)$ is a minimal model, the set of terminal valuations over $X(\s)$
coincides with the set of valuations defined by the prime exceptional divisors
of $X(\D)$ lying over the singular locus of $X(\D)$. 
These exceptional divisors are parametrized by the elements in $\D(1)_\prim \cap \s_\sing$.
By construction, this set is equal to $\Ter(\s)$. 
\end{proof}

It is clear from the toric description that $\Ter(\s) \subset \Min(\s)$.
As the following example shows, these two sets can be different. 

\begin{eg}
Let $\s$ be the cone in $\R^3$ spanned by
the vectors $v_1=(1,0,0)$, $v_2=(0,1,0)$, and $v_3=(1,1,2)$. 
Then $v = (1,1,1) \in \Min(\s) \smallsetminus \Ter(\s)$, 
and thus it gives a Nash valuation on $X(\s)$ which is not a terminal valuation.
Notice, in fact, that $X(\s)$ has terminal singularities, and so it has no terminal valuations.
\end{eg}

\subsection{Minimal valuations}

Given any variety $X$, 
consider the partial order $\le_X$ among divisorial valuations on $X$
given by declaring $\n \le_X \n'$ whenever, denoting by $p \in X$ the center of 
$\n'$, one has $0 \le \n(h) \le \n'(h)$ for all $h \in \O_{X,p}$. 
Note that $\le_X$ is the same as the order $\le_\s$
defined above when $X = X(\s)$. 
The order $\le_X$ depends on the model $X$. 
We say that a divisorial valuation $\n$ on $X$ is \emph{minimal} over $X$
if it is a minimal element, among divisorial valuations centered in the
singular locus of X, with respect to the partial order $\le_X$.

A semi-continuity property of arc valuations implies that every minimal valuation
over $X$ is a Nash valuation. 
On toric varieties, minimal valuations account for all Nash valuations, 
but this is not true in general. For instance, there is a unique minimal valuation over
an $E_8$ singularity, but there are eight Nash valuations over this singularity.

In a more general setting, every divisorial valuation $\n$ on a variety $X$ defines in a natural way
a set $C_X(\n)$ in the arc space $X_\infty$, called the \emph{maximal divisorial set}
of $\n$ (for instance, the set $N_E$ considered in \S\ref{s:outline}
is the same as the maximal divisorial set of the valuation $\ord_E$).
The problem of comparing the order $\le_X$ to the order 
given by inclusions between maximal divisorial sets
has been studied by Ishii in \cite{Ish08}, 
where it is shown that the two orders are different even when $X = \A^n$. 

Toric varieties and surfaces form two large classes of varieties
for which the Nash map is known to be surjective, but this property
seems to hold for different reasons: while in the first case
Nash valuations are all minimal, but not necessarily 
terminal, just the opposite happens in the case of surfaces. 
This seems to suggests that minimal and terminal valuations complement
each other, in some way. 
We do not know any example of a Nash valuation over a variety $X$ that is 
neither terminal nor minimal over $X$.

\appendix

\section{The curve selection lemma}

Let $X$ be a variety over $k$.
A closed, irreducible subset $N \subset X_\infty$ is said to be \emph{generically stable}
if there is an affine open set $U \subset X_\infty$ such that $N \cap U$
is not empty and is cut out, set theoretically, by finitely many 
equations (i.e., its ideal in $\O_{X_\infty}(U)$ is the radical of a finitely generated ideal). 
An arc $\a \in X_\infty$ is \emph{stable} if its closure is a generically stable set. 
There are examples of generically stable sets in $X_\infty$ that are fully contained
in $(X_\sing)_\infty$. In the following, 
we will focus on generically stable sets that are not contained in $(X_\sing)_\infty$.

If $E$ is a prime divisor on a resolution of singularities $f \colon Y \to X$, 
then the subset $N_E \subset X_\infty$ given by the closure of $f_\infty(\p_Y^{-1}(E))$
is a generically stable set and is not contained in $(X_\sing)_\infty$. 
In particular, this applies to the irreducible components of $\p_X^{-1}(X_\sing)$. 

If $\a \in X_\infty$ is a stable point that is not contained in $(X_\sing)_\infty$, 
then the completion $\^{\O_{X_\infty,\a}}$ of the local ring of $X_\infty$ at $\a$
is Noetherian (cf.~\cite[Corollary~4.6]{Reg06}). 
Using this fact, Reguera proved the following form of the curve selection lemma
for generically stable sets in arc spaces.

\begin{thm}[\protect{\cite[Corollary~4.8]{Reg06}}]
\label{t:Reg-csl}
Let $N \subsetneq M$ be the inclusion of two irreducible closed sets in $X_\infty$, and assume that
$N$ is generically stable and is not contained in 
$(X_\sing)_\infty$. Let $\a \in N$ be the generic point, and
let $\kappa_\a$ be its residue field.
Then there exist a finite algebraic extension $L/\k_\a$,
and a morphism
$\Psi \colon \Spec L[[s]] \to X_\infty$,
such that $\Psi(0) = \a$ and $\Psi(\e) \in M \smallsetminus N$.
\end{thm}

\begin{rmk}
\label{r:dim-1}
The curve selection lemma can be used to give a slick proof of
the Nash problem in dimension one.
Suppose indeed that $X$ is a curve with a singularity $p \in X$, 
and let $f \colon Y \to X$ be the normalization. 
If the Nash problem were false, then we could find a wedge 
$\Psi \colon \Spec L[[s,t]] \to X$ such that the lifts of the special arc $\Psi_0$
and the generic arc $\Psi_\e$ are based at two distinct points in the fiber over $p$. 
This is however impossible, since $\Spec L[[s,t]]$ is connected.
(There are also more elementary arguments to show this.)
\end{rmk}

In order to prove Theorem~\ref{t:main}, we need to take a suitable
specialization of the wedge produced by Theorem~\ref{t:Reg-csl}. 
This can be achieved using the results of \cite[\S2.2]{LJR12} (see also
\cite[\S3.2]{FdB12}), but for the convenience of the reader, we include a
proof. 

We start with some remarks about very general points on arc spaces. Recall
that a very general point on a scheme is a point that avoids countably many
closed subsets. The existence of such points is clear for schemes of finite
type over uncountable fields, but requires some discussion in the case of
arc spaces.

\begin{lem}
\label{l:very-gen-arc}
Let $E$ be a prime divisor on a resolution of singularities $f \colon Y \to
X$, and let $p$ be a very general $\k_p$-valued point on $E$. Then there
exists a very general $\k_p$-valued point $\b$ on $N_E$. Moreover, $\b$ can
be chosen in such a way that its lift $\~\b$ to $Y_\infty$ verifies
$\~\b(0) = p$ and $\ord_{\~\b}(E) = 1$.
\end{lem}

\begin{proof}
Let $\{H_i\}_{i \in \N}$ be a countable collection of hypersurfaces on
$X_\infty$ such that no $H_i$ contains $N_E$, and consider $\~H_i =
f_\infty^{-1}(H_i)$. We will denote by $\p_m \colon Y_m \to Y$, $\t_m
\colon Y_\infty \to Y_m$ and $\t_{m',m} \colon Y_{m'} \to Y_m$ the
truncation maps. We let $V_{i,m} \subset \p_m^{-1}(E)$ be the largest
subset for which $\t_m^{-1}(V_{i,m}) \subset \~H_i$. In other words,
$\t_m^{-1}(V_{i,m})$ is the union of all the fibers of $\t_m$ which are
contained in $\~H_i\cap\p_Y^{-1}(E)$. Notice that there exists an integer $m_0$ 
(depending on $i$) such
that for all $m \ge m_0$, the sets $V_{i,m}$ are hypersurfaces in
$\p_m^{-1}(E)$, and $\~H_i\cap\p_Y^{-1}(E) = \t_m^{-1}(V_{i,m})$. 
For $m<m_0$, since the truncation map $\t_{m_0,m}$ has irreducible fibers of
dimension $n(m_0-m)$, where $n=\dim Y$, we see that
\[
V_{i,m} = 
\{ 
\g\in\t_{m_0,m}(V_{i,m_0})
\mid 
\dim(\t_{m_0,m}^{-1}(\g)) \ge n(m_0-m)
\}.
\]
In particular, $V_{i,m}$ is constructible and its closure is a proper
subset of $\p_m^{-1}(E)$. In fact, since $\t_{m_0,m}^{-1}(V_{i,m})$ is
closed~\cite[Theorem~13.1.3]{EGA4.3} and $\t_{m_0,m}$ is smooth, we see that
$V_{i,m}$ is closed.

If necessary, we add to the collection $\{\~H_i\}_{i\in \N}$ one more
hypersurface $\~H_0$, in such a way that $\~H_0$ contains all arcs $\~\b
\in Y_\infty$ with $\ord_{\~\b}(E)\geq2$, but does not contain $\p_Y^{-1}(E)$.
Then no $\~H_i$ contains $\p_Y^{-1}(E)$, and
therefore each $V_{i,m}$ is a proper subset of $\p_m^{-1}(E)$. Since the
base field $k$ is assumed to be uncountable, we see that there exists a
point $\~\b_m$ of $\p_m^{-1}(E)$ not contained in any of the $V_{i,m}$.
Notice that $p=\~\b_0$ is a very general point of $E = \p_0^{-1}(E)$.

The result follows if we show that the $\~\b_m$ can be chosen compatibly,
i.e., in such a way that $\t_{m+1,m}(\~\b_{m+1}) = \~\b_m$. To see this,
notice that the fiber $\t_{m+1,m}^{-1}(\~\b_m)$ is isomorphic to an affine
space $\A^n_L$, where $n = \dim Y$ and $L$ is the residue field of
$\~\b_m$. We have that $\t_{m+1,m}^{-1}(\~\b_m) \not\subset V_{i,m+1}$,
since $\~\b_m \not\in V_{i,m}$. Since $L$ is uncountable, we can find an
$L$-valued point $\~\b_{m+1} \in \t_{m+1,m}^{-1}(\~\b_m) \simeq \A^n_L$ not
contained in any of the $V_{i,m+1}$, and the result follows.
\end{proof}

Let $N$ be an integral scheme with field of fractions $L$, and consider a
wedge 
\[
\Psi\colon \Spec L[[s,t]] \to X. 
\]
Given a point $\xi \in N$, we
say that $\Psi$ is \emph{defined at $\xi$} if $\Psi$ factors through $\Spec
\O_{N,\xi}[[s,t]]$. 

\begin{lem}
\label{l:specialize-csl-aux1}
If $\xi$ is a very general point of $N$, then
$\Psi$ is defined at $\xi$, and hence it induces
a wedge $\Phi \colon \Spec \k_\xi[[s,t]] \to X$
by the diagram
\[
\xymatrix{
\Spec L[[s,t]] \ar[r] \ar@/^1.3pc/[rr]^\Psi
& \Spec \O_{N,\xi}[[s,t]] \ar[r]
& X 
\\
& \Spec \k_\xi[[s,t]] \ar[u] \ar[ur]_{\Phi}
}
\]
We say that $\Phi$ is the \emph{restriction} of $\Psi$ to $\x$. 
\end{lem}

\begin{proof}
We need to show that $\Psi$ factors through $\Spec \O_{N,\xi}[[s,t]]$ for a
very general $\xi \in N$, the second assertion being a formal consequence of this.  

Pick an affine chart $W \subset X$ containing the
point $\Psi(0,0) \in X$, so that $\Psi$ factors through $W$. Fix an
embedding $W \subset \A^d$, and let $x_1,\dots,x_d$ be coordinates in
$\A^d$. Then $\Psi$ is given by equations
\[
x_\ell(s,t) = \sum_{i,j} x_{i,j,\ell} \, s^i t^j \in L[[s,t]].
\]
The result follows from the fact that $x_{i,j,\ell} \in \O_{N,\xi}$ for a
very general point $\xi \in N$. To see this, let $N^\circ$ be an affine
chart of $N$, of the form $N^\circ = \Spec R$ for some ring $R$ with field
of fractions $L$. Then we can write $x_{i,j,\ell} = a_{i,j,\ell} /
b_{i,j,\ell}$, where $a_{i,j,\ell}, b_{i,j,\ell} \in R$. We let
$H_{i,j,\ell} \subset N$ be the closure of the hypersurface in $N^\circ$
determined by $b_{i,j,\ell}$, and consider the countable collection of
closed sets $\cH = \{N\setminus N^\circ\} \cup \{ H_{i,j,\ell} \mid
i,j,\ell \}$. Then $x_{i,j,\ell} \in \O_{N,\xi}$ for all $i,j,\ell$
precisely when $\xi$ does not belong to any of the closed sets in $\cH$.
\end{proof}

We consider now a projective birational morphism $f \colon Y \to X$.
We assume that the special arc $\Psi_0$ is not fully contained in the
exceptional locus of $f$. 

\begin{lem}
\label{l:specialize-csl-aux2}
With the same notation as in Lemma~\ref{l:specialize-csl-aux1},
if $\xi$ is a very general point of $N$, then
$\Phi$ lifts to $Y$ if and only if $\Psi$ does. 
\end{lem}

\begin{proof}
We pick an affine chart $U \subset Y$
containing the point $\~\Psi_0(0)$.
For a very general choice of $\x$, the point $\~\Phi_0(0)$ will also be in $U$
(this can be seen by the same arguments as in the proof of Lemma~\ref{l:specialize-csl-aux1}).
We fix a closed embedding $U \subset \A^d$, and let $y_1, \dots, y_d$ be
coordinates in $\A^d$. 

In order to describe how $\Psi$ specializes to $\Phi$ in these coordinates, 
it is convenient to interchange the roles of $s$ and $t$, and consider the
wedge $\Psi^\vee \colon \Spec L[[s,t]] \to X$ given by 
\[
\Psi^\vee(s,t) := \Psi(t,s). 
\]
Switching $s$ and $t$ in the diagram in Lemma~\ref{l:specialize-csl-aux1}
shows that the restriction of $\Psi^\vee$ to $\x$ gives the
wedge $\Phi^\vee \colon \Spec \kappa_\x [[s,t]] \to X$ defined by $\Phi^\vee(s,t) = \Phi(t,s)$. 
Geometrically, $\Psi$ and $\Psi^\vee$ define isomorphic morphisms to $X$, 
and so do $\Phi$ and $\Phi^\vee$. In particular,
in order to prove the lemma it is equivalent
to show that, for a general choice of $\x$, 
$\Phi^\vee$ lifts to $Y$ if and only if $\Psi^\vee$ does.

The reason of looking at $\Psi^\vee$, rather than $\Psi$, is that if
$\~\Psi_\e^\vee$ is the lift of the generic arc $\Psi_\e^\vee$ of $\Psi^\vee$, 
then $\~\Psi_0(0)$ is a specialization of $\~\Psi_\e^\vee(0)$, 
and therefore $\~\Psi_\e^\vee$ factors through $U$. 
Similarly, $\~\Phi^\vee_\e$ factors through $U$, and we have the diagram
\[
\xymatrix{
\Spec L((s)) \ar[r] \ar@/^1.3pc/[rr]^{\~\Psi_\e^\vee}
& \Spec \O_{N,\xi}((s)) \ar[r]
& U_\infty
\\
& \Spec \k_\xi((s)) \ar[u] \ar[ur]_{\~\Phi_\e^\vee}
}
\]
If $N^\circ = \Spec R$ is an affine chart in $N$, 
then the lifts $\~\Psi_\e^\vee$ and $\~\Phi_\e^\vee$ are given by equations
\[
y_\ell(s,t) = \sum_{i,j} 
\frac{a_{i,j,\ell}}{b_{i,j,\ell}} \, s^i t^j \in L((s))[[t]],
\and
y_\ell(s,t) = \sum_{i,j} 
\frac{\ov a_{i,j,\ell}}{\ov b_{i,j,\ell}} \, s^i t^j \in
\k_\xi((s))[[t]],
\]
where $a_{i,j,\ell}, b_{i,j,\ell} \in R$, and $\ov a_{i,j,\ell}$ and $\ov
b_{i,j,\ell}$ are the images of $a_{i,j,\ell}$ and $b_{i,j,\ell}$ in
$\k_\xi$. 

Note that if either one of $\Psi^\vee$ and $\Phi^\vee$ lifts to $Y$, 
then its lift factors through $U$. 
Then Lemma~\ref{l:lift-aux} implies that 
$\Psi$ (resp.,~$\Phi$) lifts to $Y$ if and only if $a_{i,j,\ell} = 0$ for
all $i<0$ (resp.,~$\ov a_{i,j,\ell} = 0$ for all $i<0$). In particular, if
$\Psi$ lifts, so does $\Phi$. Conversely, if $\Psi$ does not lift, we have some
non-zero element $a_{i,j,\ell}$ with $i<0$ which defines a hypersurface in
$N^\circ$, and, for any $\xi$ not belonging to this hypersurface, the
corresponding $\Phi$ does not lift.
\end{proof}

\begin{thm}
\label{t:specialize-csl}
Let $f \colon Y \to X$ be a projective birational morphism of varieties, and
let $\a \in X_\infty$ be an arc that is not fully contained in $(X_\sing)_\infty$. 
Suppose that 
\[
\Psi \colon \Spec L[[s]] \to X_\infty
\]
is a morphism such that $L/\k_\a$ is a finite algebraic extension of the
residue field of $\a$, and $\Psi(0) = \a$.
Let $N \subset X_\infty$ be the closure of $\a$, and let
$\b$ be a very general point of $N$.
Then there is a finite algebraic extension $K/\k_\b$, and a morphism
\[
\Phi \colon \Spec K[[s]] \to X_\infty,
\]
such that $\Phi(0) = \b$. 
Moreover, $\Phi$ factors through $Y_\infty$ if and only if $\Psi$ does. 
\end{thm}

\begin{proof}
We consider generators $\zeta_1, \ldots, \zeta_l$ for the field extension
$L/\k_\a$. We pick an affine chart $N^\circ=\Spec R \subset N$, where $R$
is a domain with field of fractions $\k_\a$, and $\zeta_1, \ldots, \zeta_l$
are integral over $R$. If $R'$ is the $R$-algebra generated by $\zeta_1,
\ldots, \zeta_l$, and we consider $N'= \Spec R'$, then the natural map $\r
\colon N' \to N^\circ$ is finite. Notice that $L$ is the field of fractions
of $R'$. 

Since $R'$ is integral over $R$, the image $\r(H)$ of a proper closed
subset $H \subsetneq N'$ is a proper closed subset of $N^\circ$. Therefore,
the image of a very general point in $N'$ is a very general point of $N$,
and a very general point in $N$ is the image of very general point in $N'$.
For our very general point $\b \in N$, we let $\b'$ denote a very general
point in $N'$ such that $\r(\b') = \b$. We set $K = \k_{\b'}$. Since the
map $\r$ is finite, we see that $K$ is a finite extension of $\k_\b$.

By Lemma~\ref{l:specialize-csl-aux1}, the wedge $\Psi$ is defined at
$\b'$, and restricts to a wedge $\Phi \colon \Spec K[[s]] \to X_\infty$.
By construction, it is clear that $\Phi(0) = \b$. Moreover, by
Lemma~\ref{l:specialize-csl-aux2}, the wedge $\Psi$ lifts to $Y$ if and
only if $\Psi$ does, as required.
\end{proof}

\begin{proof}[Proof of Theorem~\ref{t:csl}]
The assumption of the theorem means that the closure $N_E$ of $\a$
is strictly contained in an irreducible component $M$ of $\p_X^{-1}(X_\sing)$.
Note, on the other hand, that $N_E \not\subset (X_\sing)_\infty$. 
Then, denoting by $\k_\a$ the residue field of $\a$, 
Theorem~\ref{t:Reg-csl} implies that
there is a finite algebraic field extension $L/\k_\a$, 
and a morphism $\Psi \colon L[[s]] \to X_\infty$,
such that $\Psi(0) = \a$ and $\Psi(\e) \in M \smallsetminus N_E$. 
Note, in particular, that $\Psi$ does not lift to $Y$, 
since $\~\a$ is the generic point of an irreducible component
of $\p_Y^{-1}(f_\infty^{-1}(X_\sing))$.  
By Lemma~\ref{l:very-gen-arc}, we can find a very general $\k_p$-valued point
$\b \in N_E$ that lifts to a $\k_p$-valued point $\~\b \in
Y_\infty$ such that $\~\b(0) = p$ and $\ord_{\~\b}(E) = 1$.
Then the assertion follows from Theorem~\ref{t:specialize-csl}.
\end{proof}


\end{document}